\documentclass{article}
\usepackage{amsmath,amssymb}
\usepackage{graphicx}
\usepackage{hyperref}
\usepackage{mathtools}
\usepackage[thmmarks]{ntheorem}
\usepackage[utf8]{inputenc}

\newtheorem{remark}[subsection]{Remark}
\newtheorem{lemma}[subsection]{Lemma}
\newtheorem{theorem}[subsection]{Theorem}
\newtheorem{theorem+definition}[subsection]{Theorem \& Definition}
\theorembodyfont{\upshape}
\newtheorem{definition}[subsection]{Definition}

\theoremstyle{nonumberplain}
\theoremsymbol{\ensuremath{\square}}
\newtheorem{proof}{Proof}

\newcommand\N{\mathbb{N}}

\newcommand\PP{\mathbb{P}}
\newcommand\OO{\mathbb{O}}

\DeclareMathOperator{\Area}{Area}
\DeclareMathOperator{\Dist}{Dist}

\DeclarePairedDelimiter\abs{\lvert}{\rvert}

\title{Detecting Relatively Quasiconvex Subgroups and Their Induced Peripheral Structure}
\author{Thomas Carstensen \\ Kiel University \\\href{mailto:carstensen@math.uni-kiel.de}{carstensen@math.uni-kiel.de}}
\date{}

\begin{document}

\maketitle

\begin{abstract}
This paper proves under certain conditions the existence of an algorithm, which detects relatively quasiconvex subgroups $H$ of relatively hyperbolic groups $(G,\PP)$. Additionally, this algorithm outputs an induced peripheral structure of $(G,\PP)$ on $H$.
\end{abstract}

\section{Introduction}\label{sec:Introduction}

It is well-known that many algorithmic problems are solvable for hyperbolic groups. Examples are the computability of the hyperbolicity constant for a given finite generating set of the group \cite{Papasoglu1996}, the word problem and the conjugacy problem.

Let $G$ be a hyperbolic group with finite generating set $X$ and $H$ some subgroup of $G$. $H$ is called a \emph{quasiconvex subgroup} of $G$, if there is some $\nu>0$ such that any geodesic $v$ between two elements of $H$ in the Cayley graph $\Gamma(G,X)$ lies within a $d_X$-distance of at most $\nu$ from $H$. $\nu$ is then called a \emph{quasiconvexity constant} for $H$.

Quasiconvex subgroups are of interest, because they inherit the geometric structure from the ambient hyperbolic group. They are, for example, hyperbolic groups themselves.

Kapovich \cite{Kapovich1996} provided a semi-algorithm, which given a finite presentation of a hyperbolic group $G$ and a finite generating set of the subgroup $H$ stops if and only if $H$ is a quasiconvex subgroup of $G$ and if it stops it returns the quasiconvexity constant of $H$.

\smallskip
A common generalization of a hyperbolic group is a group that is hyperbolic relative to some finite collection of subgroups. Such a finite collection of subgroups is often referred to as a peripheral structure of the group. (For definitions see Section~\ref{sec:Preliminaries}.) It turns out that many problems, which are solvable for hyperbolic groups, have a solution for relatively hyperbolic groups as well, provided the parabolic subgroups are reasonably well behaved (cf.~\cite{Dahmani2008}, \cite{Farb1998}, \cite{Bumagin2004}).

It is also possible to generalize the concept of quasiconvex subgroups to relatively hyperbolic groups, leading to relatively quasiconvex subgroups (cf.~\cite{Hruska2010}). As described in Definition~\ref{def:rel qc}, this definition again hinges on the existence of some constant $\nu$, which bounds how far relative geodesics between elements of the subgroup can stray from it. A relatively quasiconvex subgroup is again itself hyperbolic relative to some peripheral structure, which is induced by the peripheral structure of the surrounding group and is called an induced structure.

\smallskip
Kharlampovich et al.~\cite{Kharlampovich2017} gave an algorithm, which detects relatively quasiconvex subgroups of toral relatively hyperbolic groups with peripherally finite index. Toral relatively hyperbolic groups are torsion-free relatively hyperbolic groups with abelian parabolic subgroups and a relatively quasiconvex subgroup has peripherally finite index, if its induced structure consists of conjugates of finite index subgroups of parabolic groups.

To this end, they show in a more general case that it is possible to find a language $L$ on a relatively hyperbolic group such that $L$-quasiconvex subgroups are relatively quasiconvex and such that relatively quasiconvex subgroups with peripherally finite index are $L$-quasiconvex (cf.~\cite{Kharlampovich2017},~Thm.~7.5). Using Stallings-graphs they are able to then show that there is an algorithm detecting $L$-quasiconvex subgroups.

Generalizing this result, Kim~\cite{Kim2021} proved that there is an algorithm detecting any finitely generated relatively quasiconvex subgroup of a toral relatively hyperbolic group. The basis of this generalization is a result by Manning and Mart\'inez-Pedrosa~\cite{Manning2010}, which implies that a subgroup $H$ of a toral relatively hyperbolic group is relatively quasiconvex, if and only if there is a relatively quasiconvex subgroup with peripherally finite index, which is an amalgamated product of $H$ and its induced structure.

\smallskip
The main theorem of this paper provides a more general result using independent and more geometric methods of proof.

\begin{theorem}
Let $G$ be a relatively hyperbolic group with peripheral structure $\PP=\{P_1,\ldots,P_n\}$. Let $X$ be a symmetric finite generating set of $G$. Suppose a finite relative presentation of $(G,\PP)$ and finite presentations for every $P_i$ are given. Let $H$ be a subgroup of $G$ with finite generating set $Y\subseteq X^*$. 

Suppose further that the following are given:
\begin{enumerate}
\item a solution to the membership problem of $(P_i,O)$ for every finitely generated subgroup $O\leq P_i$ and
\item an algorithm, which decides whether a given finitely generated subgroup of some $P_i$ is finite.
\end{enumerate}

Then there is a semi-algorithm which stops, if and only if $H$ is a relatively quasiconvex subgroup of $G$, and if it stops, it returns
\begin{itemize}
\item generating sets of an induced structure $\OO$ on $H$, and
\item $\lambda\geq1,c\geq0$, such that the inclusion $(H,d_{Y\cup\mathcal{O}})\to(G,d_{X\cup\mathcal{P}})$ is a $(\lambda,c)$-quasiisometric embedding.
\end{itemize}
\end{theorem}

It is an easy generalization of a theorem by Osin \cite{Osin2006}, that the relative quasiconvexity constant $\nu$ for $H$ in $G$ can be calculated from the outputs of this semi-algorithm. (Theorem~\ref{thm:qc-embedding->rel qc})

\smallskip
Section~\ref{sec:Preliminaries} of this paper will introduce necessary preliminaries and some algorithmic results that are generally helpful when dealing with relatively hyperbolic groups. Section~\ref{sec:Distortion} gives a short overview of distortion functions, which play a crucial role in the construction of the semi-algorithm in the main theorem. Section~\ref{sec:Algorithm} is then dedicated to proving the main theorem as well as some auxiliary results that are used in the proof. Section~\ref{sec:Discussion} gives a short discussion of the main theorem, in particular how its assumptions might be relaxed and a sketch of a more constructive algorithm which might turn out to be more efficient for an actual implementation.

\section{Preliminaries}\label{sec:Preliminaries}

This section gives definitions for some fundamental concepts such as relatively hyperbolic groups and relatively quasiconvex subgroups, as well as important results that provide basic tools to algorithmically investigate these groups.

Note that throughout the rest of this paper two different notions of length are used. For some set $X$ and some word $w\in X^\ast$ over $X$, $\abs{w}$ will denote the number of letters in $w$.

For some group $G$ with generating set $X$ and some element $g\in G$, $\abs{g}_X$ denotes the word length of $g$ with respect to $X$, i.e.\ the length of some shortest word in $X^\ast$ representing $g$.

In a common abuse of notation, a word $w\in X^\ast$ will often be identified with the element which it represents. So, while $\abs{w}$ denotes the length of $w$, $\abs{w}_X$ denotes the word length with respect to $X$ of the element represented by $w$.

The first important definition is that of a relatively hyperbolic group. The following definition is due to Osin (\cite{Osin2006}, Definition~2.35).
\begin{definition}
Let $G$ be a group. A finite set $\PP=\{P_1,\ldots,P_n\}$ of subgroups of $G$ is called a \emph{peripheral structure of $G$}.

Let $\langle X,P_1,\ldots,P_n\mid R\rangle$ be a finite relative presentation of $(G,\PP)$. For each $i\in\{1,\ldots,n\}$ let $\tilde{P}_i$ be an isomorphic copy of $P_i$, such that $\tilde{P}_1,\ldots,\tilde{P}_n,X$ are mutually disjoint. Let $\mathcal{P}:=\bigcup (\tilde{P}_i\setminus\{1\})$.

$G$ is \emph{hyperbolic relative to $\PP$}, if the relative Dehn function of this presentation is linear, i.e.\ if there exists $C>0$, such that any word $w\in(X\cup\mathcal{P})^\ast$ of length at most $l$ representing the trivial element in $G$ can be written as the product of at most $C\cdot l$ conjugates of elements of $R$ in the free product $F(X)\ast(\ast_{i\in\{1,\ldots,n\}}\tilde{P}_i)$.
\end{definition}

A natural class of subgroups of a relatively hyperbolic group $(G,\PP)$ are the relatively quasiconvex subgroups.

\begin{definition}[Osin \cite{Osin2006}, Def.~4.9]\label{def:rel qc}
Let $(G,\PP)$ be a re\-la\-tive\-ly hyperbolic group and $X$ a finite generating system of $G$. $H\leq G$ is called a \emph{relatively quasiconvex subgroup} of $G$, if there exists $\nu\geq0$, such that every vertex of a geodesic $p$ in $\Gamma(G,X\cup\mathcal{P})$ with endpoints in $H$ has at most a $d_X$-distance of $\nu$ from $H$.
\end{definition}

Relatively quasiconvex subgroups inherit the geometric structure of $G$ and are themselves hyperbolic relative to a peripheral structure that is induced by $\PP$.

\begin{theorem+definition}[Hruska \cite{Hruska2010}, Thm.~9.1]\label{thm+def:induced structure}
Let $(G,\PP)$ be re\-la\-tive\-ly hyperbolic and $H$ a relatively quasiconvex subgroup of $G$. 

Then the subgroups in
\[
\bar{\OO}=\{H\cap P^g\mid g\in G,P\in\PP,\abs{H\cap P^g}=\infty\}
\]
lie in finitely many conjugacy classes in $H$. For a set of representatives $\OO$ of these conjugacy classes, $(H,\OO)$ is relatively hyperbolic. The peripheral structure $\OO$ is called an \emph{induced structure} of $(G,\PP)$ on $H$.

Moreover, the inclusion $(H,d_{Y\cup\mathcal{O}})\to(G,d_{X\cup\mathcal{P}})$ is a quasiisometric embedding for any finite relative generating set $Y$ of $(H,\OO)$.
\end{theorem+definition}

The following algorithmic result by Dahmani \cite{Dahmani2008} gives conditions for when it is possible to compute an upper bound for the relative Dehn function as well as a hyperbolicity constant for the Cayley graph $\Gamma(G,X\cup\mathcal{P})$.

\begin{theorem}[Dahmani \cite{Dahmani2008}]\label{thm:dehn+delta}
Let $G$ be a group with some peripheral structure $\PP=\{P_1,\ldots,P_n\}$ consisting of finitely presented subgroups. Suppose each $P_i$ has solvable word problem. Let $\langle X,P_1,\ldots,P_n\mid R\rangle$ be a finite relative presentation of $G$.

Then there is an semi-algorithm that stops if and only if the relative presentation of $G$ satisfies a linear isoperimetric inequality. If it stops, it provides:
\begin{enumerate}
\item $K>0$ such that $\Area(w)<K\abs{w}$ for all relators $w$.
\item $\delta\geq0$ such that $\Gamma(G,X\cup\mathcal{P})$ is $\delta$-hyperbolic.
\end{enumerate}
\end{theorem}

An important restriction of this theorem is that all peripheral subgroups must have a solvable word problem. Moreover, Dahmani gives a concrete algorithm as in the theorem under the condition that solutions to the word problems of all $P_i$ are given.

The following terminology was introduced by Osin \cite{Osin2006}:

\begin{definition}
A subpath of a path $p$ in the Cayley graph $\Gamma(G,X\cup\mathcal{P})$ is called a \emph{$P_i$-component} of $p$, if it is a maximal non-trivial subpath, which is labeled by letters in $P_i$.

A path is \emph{locally minimal} if every $P_i$-component has length $1$.

Two $P_i$-components are said to be \emph{connected} if there is an edge labeled by an element of $P_i$ connecting them. A $P_i$-component of a path $p$ is \emph{isolated}, if it is not connected to some other $P_i$-component of $p$.

A path $p$ is \emph{without backtracking} if all its components are isolated.
\end{definition}

Note that Osin introduces different but analogue terminology for words in $(X\cup\mathcal{P})^\ast$. It can be interpreted as the result of identifying a word $w\in(X\cup\mathcal{P})^\ast$ with the path in the relative Cayley graph $\Gamma(G,X\cup\mathcal{P})$ starting at $1$ and labeled by $w$.

For the sake of simplicity this distinction will be ignored, with the understanding that the given terminology applies through this identification in the obvious way.

In addition, by this identification a word $w\in(X\cup\mathcal{P})^\ast$ is said to be quasigeodesic, if the corresponding path in $\Gamma(G,X\cup\mathcal{P})$ is quasigeodesic.

Osin proved that relatively hyperbolic groups fulfill the BCP-property introduced by Farb \cite{Farb1998}.

\begin{theorem}[cf.~Osin~\cite{Osin2006}, Thm.~3.23]\label{thm:bcp}
Let $G$ be a group with symmetric finite generating set $X$. Let $G$ be hyperbolic relative to some peripheral structure $\PP=\{P_1,\ldots,P_n\}$ of finitely presented subgroups. Let $G$ be finitely presented relative $\PP$. Suppose the word problem in each $P_i$ is solvable. Let $\lambda\geq1$, $c\geq0$.

\smallskip
Then there is an algorithm which calculates $\varepsilon=\varepsilon(\lambda,c)>0$ such that for any locally minimal $(\lambda,c)$-quasi-geodesics $p,q$ with the same endpoints and without backtracking in $\Gamma(G,X\cup\mathcal{P})$ the following is true:
\begin{itemize}
\item[i)] Every vertex of $p$ has an $d_X$-distance of at most $\varepsilon$ from a vertex of $q$.

\item[ii)] For each $P_i$-component $s$ of $p$ with $d_X(s^-,s^+)>\varepsilon$ there is a $P_i$-component $t$ of $q$ that is connected to $s$.

\item[iii)] For connected $P_i$-components $s$ of $p$ and $t$ of $q$ the following holds:
\[
d_X(s^-,t^-),d_X(s^+,t^+)<\varepsilon
\]
\end{itemize}
\end{theorem}

While Osin does not give an explicit algorithm to compute $\varepsilon$, it is easy to see that his arguments can be rewritten to derive such an algorithm. Osin uses that the relative Dehn function is bounded by a linear function. To construct this algorithm it is therefore necessary to calculate such an upper bound. To this end, the assumptions that every $P_i$ is finitely presented and has solvable word problem were added in the above theorem. By Theorem~\ref{thm:dehn+delta} this suffices to be able to compute a linear upper bound for the relative Dehn function.

Now let $Y\subseteq X^*$ be a generating set of some subgroup $H\leq G$ and let $\OO=\{O_1,\ldots,O_m\}$ where each $O_j$ is a subgroup of $H\cap P_{i_j}^{g_j}$ with $1\leq i_j\leq n$ and $g_j\in X^*$. Then there is a canonical way of mapping each word in $(Y\cup\mathcal{O})^*$ to some word in $(X\cup\mathcal{P})^*$ representing the same group element.

\begin{definition}\label{def:canonical embedding}
The \emph{canonical map} $\iota\colon(Y\cup\mathcal{O})^*\to(X\cup\mathcal{P})^*$ is defined as the canonical extension of the map $Y\cup\mathcal{O}\to(X\cup\mathcal{P})^*$ that maps each element of $Y\subseteq(X\cup\mathcal{P})^*$ to itself and that maps every $o\in O_j$ to the unique word $g_j^{-1}pg_j\in(X\cup\mathcal{P})^*$ representing $o$ with $p\in P_{i_j}$.
\end{definition}

A finitely generated subgroup of a hyperbolic group is quasiconvex, if and only if the inclusion of the subgroup is a quasiisometric embedding. Theorem~\&~Definition~\ref{thm+def:induced structure} generalizes the forward implication by stating that for a relatively quasiconvex subgroup $H$ of $G$ with induced structure $\OO$ the inclusion $(H,d_{Y\cup\mathcal{O}})\to(G,d_{X\cup\mathcal{P}})$ is a quasiisometric embedding.

The following theorem and its proof are straightforward generalizations of a theorem by Osin (\cite{Osin2006}, Thm.~4.13) and can be viewed as generalizations of the converse implication in the above statement.

\begin{theorem}\label{thm:qc-embedding->rel qc}
Let $G$ be a group with symmetric finite generating set $X$ and hyperbolic relative to a collection $\PP=\{P_1,\ldots,P_n\}$ of subgroups of $G$. Let $H\leq G$ with finite generating set $Y\subseteq X^*$.

Let $\OO=\{O_1,\ldots,O_m\}$ such that for each $1\leq j\leq m$ there are $1\leq i_j\leq n$ and $g_j\in X^*$ with $O_j\subseteq H\cap P_{i_j}^{g_j}$.

Suppose the inclusion $(H,d_{Y\cup\mathcal{O}})\to(G,d_{X\cup\mathcal{P}})$ is a $(\lambda,c)$-quasiisometric embedding.

Then $H$ is relatively quasiconvex in $G$ and it is possible to compute a constant $\nu$ as in Definition~\ref{def:rel qc}.

\begin{proof}
Let
\[
\mu:=\max(\{\abs{y}\mid y\in Y\}\cup\{2\abs{g_j}+1\mid 1\leq j\leq m\}).
\]
Let $h\in H$ and $V=z_1\ldots z_l\in(Y\cup\mathcal{O})^*$ be a minimal word representing $h$, i.e.\ $l=\abs{h}_{Y\cup\mathcal{O}}$.

Define $U:=\iota(V)$, and let $U_0$ be a subword of $U$, i.e.
\[
U_0=A\iota(z_r)\ldots\iota(z_{r+s})B
\]
with $\abs{A},\abs{B}\leq\mu$.

Since every subword of $V$ is geodesic, the following holds:
\begin{align*}
\abs{U_0} &\leq 2\mu+\mu(s+1)\\
&= 2\mu+\mu\abs{z_r\ldots z_{r+s}}_{Y\cup\mathcal{O}}\\
&\leq 2\mu+\mu(\lambda\abs{z_r\ldots z_{r+s}}_{X\cup\mathcal{P}}+\lambda c)\\
&\leq 2\mu+\mu(\lambda(\abs{U_0}_{X\cup\mathcal{P}}+2\mu)+\lambda c)
\end{align*}

Therefore, the path $p$ in $\Gamma(G,X\cup\mathcal{P})$ from $1$ to $h$, which is labeled by $U$, is a $(\mu\lambda,2\mu+2\mu^2\lambda+\mu\lambda c)$-quasigeodesic.

Assume $p$ has connected $P_i$-components. Then there is some subword $p_1wp_2$ of $U$ for some $w\in(X\cup\mathcal{P})^\ast$ and $p_1,p_2\in P_i$, which represents an element $p_3\in P_i$. Replace the subpath of $p$ labeled by $p_1wp_2$ with the single edge connecting its endpoints and labeled by $p_3$. This new path is still a $(\mu\lambda,2\mu+2\mu^2\lambda+\mu\lambda c)$-quasigeodesic and its vertex set is a subset of the vertex set of $p$.

Repeating this process eventually yields a $(\mu\lambda,2\mu+2\mu^2\lambda+\mu\lambda c)$-quasigeodesic $\bar{p}$ without backtracking, whose vertex set is a subset of the vertex set of $p$.

For any geodesic $q$ in $\Gamma(G,X\cup\mathcal{P})$ from $1$ to $h$ and any vertex $v$ of $q$ it follows with $\varepsilon=\varepsilon(\mu\lambda,2\mu+2\mu^2\lambda+\mu\lambda c)$ as in the conclusion of Theorem \ref{thm:bcp}, that there exists a vertex $u$ of $\bar{p}$ such that:
\[
d_X(u,v)\leq\varepsilon.
\]

Since $d_X(u,H)\leq\mu$ for any vertex $u$ of $p$ and therefore in particular for any vertex of $\bar{p}$, it follows that $d_X(v,H)\leq\varepsilon+\mu$. Hence, $H$ is relatively quasiconvex as in Definition~\ref{def:rel qc} with $\nu:=\varepsilon+\mu$.
\end{proof}
\end{theorem}

\begin{remark}\label{rem:not induced}
The inclusion $(H,d_{Y\cup\mathcal{O}})\to(G,d_{X\cup\mathcal{P}})$ is by Theorem~\&~Definition~\ref{thm+def:induced structure} a quasiisometric embedding if $\OO$ is some induced structure on the relatively quasiconvex subgroup $H$ of $G$.

It is however not the case that any structure $\OO$ as in Theorem~\ref{thm:qc-embedding->rel qc} for which this inclusion is a quasiisometric embedding has to be an induced structure. In fact, let $G$ be a group with finite generating set $X$ and hyperbolic relative to $\PP=\{P_1,\ldots,P_n\}$, and $\OO:=\{O_1,\ldots,O_n\}$, where each $O_i$ is a finite index subgroup of $P_i$ and at least one $O_i$ is a proper subgroup of $P_i$. Then $G$ is not hyperbolic relative $\OO$, since $\OO$ is not an almost malnormal collection (cf.~Osin~\cite{Osin2006}, Prop.~2.36), but the inclusion $(G,d_{X\cup\mathcal{O}})\to(G,d_{X\cup\mathcal{P}})$ is a quasiisometric embedding.
\end{remark}

\section{Distortion}\label{sec:Distortion}

Distortion functions are an important tool in proving the main theorem. This section contains their definition and basic results which are used in Section \ref{sec:Algorithm}.

\begin{definition}
Let $G$ be a group and $H\leq G$ a subgroup. Let $X$ be a finite generating set of $G$ and $Y$ a finite subset of $G$ such that $H\leq\langle Y\rangle$. Then
\[
\Dist_{(G,X)}^{(H,Y)}\colon\N\to\N,\,n\mapsto\max\{\abs{h}_Y\mid h\in H,\,\abs{h}_X\leq n\}
\]
is the \emph{distortion function of $H$ in $G$ with respect to $Y$ and $X$}.
\end{definition}

This definition somewhat generalizes the usual definition of distortion functions. It does not require the set $Y$ to be a generating set of the subgroup $H$. Instead it allows for $Y$ to be a generating set of some subgroup of $G$ containing $H$. In this sense it measures the distortion between two induced word metrics on $H$.

The following two lemmas are straightforward consequences of the definition:

\begin{lemma}\label{lem:distortion facts}
Let $G$ be a group and $H\leq G$. Let $X$ be a finite generating set of $G$ and $Y$ a finite subset of $G$ such that $H\leq\langle Y\rangle$. Then the following holds:
\begin{enumerate}
\item Let $Y'$ be a finite subset of $G$ such that $H\leq\langle Y'\rangle\leq\langle Y\rangle$. Then:
\[
\Dist_{(G,X)}^{(H,Y)}\leq\max\{\abs{y'}_Y\mid y'\in Y'\}\cdot\Dist_{(G,X)}^{(H,Y')}
\]

\item Let $g\in X^*$. Then for all $n\in\N$:
\[
\Dist_{(G,X)}^{(H^g,Y^g)}(n)\leq\Dist_{(G,X)}^{(H,Y)}(n+2\abs{g})
\]
\end{enumerate}
\end{lemma}

\begin{lemma}\label{lem:chained distortion}
Let $G$ be a group and $H_2\leq H_1\leq G$ subgroups. Let $X,Y_1,Y_2$ be finite generating sets of $G$, $H_1$ and $H_2$ respectively. Then:
\[
\Dist_{(G,X)}^{(H_2,Y_2)}\leq\Dist_{(H_1,Y_1)}^{(H_2,Y_2)}\circ\Dist_{(G,X)}^{(H_1,Y_1)}
\]
\begin{proof}
Let $n\in\N$ and $h\in H_2$ with $\abs{h}_X\leq n$. Since distortion functions are monotonous, the following holds:
\begin{align*}
\abs{h}_{Y_2} &\leq \Dist_{(H_1,Y_1)}^{(H_2,Y_2)}(\abs{h}_{Y_1})\\
&\leq \Dist_{(H_1,Y_1)}^{(H_2,Y_2)}\left(\Dist_{(G,X)}^{(H_1,Y_1)}(\abs{h}_X)\right)\\
&\leq \Dist_{(H_1,Y_1)}^{(H_2,Y_2)}\circ\Dist_{(G,X)}^{(H_1,Y_1)}(n)
\end{align*}\hspace*{0pt}
\end{proof}
\end{lemma}

It is well-known (cf.~\cite{Drutu2005}) that maximal parabolic subgroups of a finitely generated relatively hyperbolic group are undistorted, i.e.\ have a linear distortion function. In fact, a result by Osin~(\cite{Osin2006},~Lemma~5.4) states that, up to some computable constants, the distortion function of a peripheral subgroup is bounded by the relative Dehn function. This implies:

\begin{theorem}\label{thm:undistorted parabolics}
Let $G$ be a group which is finitely presented relative to some peripheral structure $\PP=\{P_1,\ldots,P_n\}$. Let $X$ and $X_i$ be finite generating sets for $G$ and each $P_i$ respectively.

Suppose the relative Dehn function of $(G,\PP)$ is computable.

Then $\Dist_{(G,X)}^{(P_i,X_i)}$ is computable.
\end{theorem}

The final theorem in this section gives a connection between the computation of a distortion function and the solution to the corresponding membership problem. (cf.~Farb~\cite{Farb1994}, Proposition~2.1)

\begin{theorem}\label{thm:memb prob<->distortion}
Let $G$ be a group and $H\leq G$. Let $X$ be a finite generating set of $G$ and $Y\subseteq X^*$ finite, such that $H\leq\langle Y\rangle$.
Suppose a solution to the word problem of $G$ is given. Then
\begin{enumerate}
\item given a solution to the membership problem of $(G,H)$ it is possible to compute $\Dist_{(G,X)}^{(H,Y)}$, and

\item given an algorithm that computes $\Dist_{(G,X)}^{(H,Y)}$ it is possible to construct a solution to the membership problem of $(G,H)$.
\end{enumerate}
\begin{proof}Suppose a solution to the membership problem of $(G,H)$ is given. Let $n\in\N$.

Using the solution of the membership problem of $(G,H)$, it is possible to compute the set
\[
W:=\{w\in X^*\mid \abs{w}\leq n,\,w\in H\},
\]
which contains representatives of all elements $h\in H$ with $\abs{h}_X\leq n$.

With the solution to the word problem in $G$ it is now possible to determine the $Y$-length of all elements that are represented by a word in $W$.

The maximum of all these lengths is precisely $\Dist_{(G,X)}^{(H,Y)}(n)$.

Suppose now an algorithm computing $\Dist_{(G,X)}^{(H,Y)}$ is given. Let $w\in X^*$.

By definition of $\Dist_{(G,X)}^{(H,Y)}$, the element represented by $w$ lies in $H$, if and only if there is a word $w'\in Y^*$ representing this element with $\abs{w'}\leq\Dist_{(G,X)}^{(H,Y)}(\abs{w})$.

Since a solution to the word problem in $G$ is given, it is possible to decide for all $w'\in Y^*$ with $\abs{w'}\leq\Dist_{(G,X)}^{(H,Y)}(\abs{w})$, whether $\iota(w')w^{-1}$ represents the trivial element in $G$.
\end{proof}
\end{theorem}

\section{Algorithm}\label{sec:Algorithm}

In this section the following is always assumed to be given:

$G$ is a group with symmetric finite generating set $X$ and $G$ is relatively hyperbolic with respect to a collection of subgroups $\PP=\{P_1,\ldots,P_n\}$. There is a finite relative presentation $\langle X,P_1,\ldots,P_n\mid R\rangle$ of $(G,\PP)$. Each $P_i$ has a finite presentation $\langle X_i\mid R_i\rangle$, where $X_i$ is given as a subset of $X^*$. $H$ is a subgroup of $G$ with finite generating set $Y\subseteq X^*$.

The goal of this section is to prove the main theorem by giving an explicit algorithm. The core idea of this algorithm will be to successively check for all possible peripheral structures $\OO$ of $H$ as in Theorem~\ref{thm:qc-embedding->rel qc}, whether the inclusion $(H,d_{Y\cup\mathcal{O}})\to(G,d_{X\cup\mathcal{P}})$ is a quasiisometric embedding. If such a structure $\OO$ exists, then by Theorem~\ref{thm:qc-embedding->rel qc}, $H$ is relatively quasiconvex in $G$.

However, as mentioned in Remark~\ref{rem:not induced}, this structure $\OO$ does not always need to be an induced structure. It turns out that to verify both that the above inclusion is a quasiisometric embedding as well as that $\OO$ is an induced structure it is helpful to check for the existence of elements as described in the following lemma. Their existence is an obstruction to $\OO$ being an induced structure.

\begin{lemma}\label{lem:enlarge parabolics}
Let $\OO=\{O_1,\ldots,O_m\}$ be a peripheral structure of $H$ consisting of subgroups $O_j\subseteq H\cap P_{i_j}^{g_j}$ for some $1\leq i_j\leq n$ and $g_j\in X^*$.

Suppose for $j,k\in\{1,\ldots,m\}$ with $i:=i_j=i_k$ there is $h\in H\cap g_j^{-1}P_ig_k$ with $h\notin O_j$ if $j=k$.

Then
\begin{itemize}
\item $O_j$ or $O_k$ is not maximal parabolic, or
\item $j\neq k$ and $O_j$ and $O_k$ are conjugated in $H$.
\end{itemize}
\begin{proof}
If $j=k$, this is obvious. In this case, $h\notin O_j$ and therefore $O_j$ is a proper subgroup of the parabolic group $\langle O_j,h\rangle\subseteq H\cap P_i^{g_j}$, i.e.\ $O_j$ is not maximal parabolic.

If on the other hand $j\neq k$ and $h=g_j^{-1}pg_k$ for some $p\in P_i$, then:
\[
O_j^h\subseteq(H\cap P_i^{g_j})^h=H^h\cap P_i^{pg_k}=H\cap P_i^{g_k}
\]

This implies that $\langle O_j^h,O_k\rangle$ is a subgroup of $H\cap P_i^{g_k}$. So either $O_j$ or $O_k$ are not maximal parabolic or $O_j^h=O_k$.
\end{proof}
\end{lemma}

To verify that the inclusion $(H,d_{Y\cup\mathcal{O}})\to(G,d_{X\cup\mathcal{P}})$ is a quasiisometric embedding it is sufficient to verify that all geodesic words $w\in(Y\cup\mathcal{O})^*$ are mapped to uniformly quasigeodesic words $\iota(w)\in(X\cup\mathcal{P})^*$ under the canonical map $\iota$. Since $(Y\cup\mathcal{O})^*$ contains words of arbitrary length over an infinite alphabet, it is not immediately clear how to verify this property algorithmically.

The reduction to words of bounded length is easy due to the well-known property of hyperbolic spaces that for appropriate constants, local quasigeodesics are quasigeodesics. Hence, it suffices to show for appropriate constants $L$, $\lambda$ and $c$ that geodesic words of length at most $L$ are mapped to $(\lambda,c)$-quasigeodesic words.

The problem of the infinite alphabet can be handled by an application of Theorem~\ref{thm:bcp}. It implies, that for any two quasigeodesic words in $(X\cup\mathcal{P})^*$ representing the same element and without backtracking every long $P_i$-component of one of the quasigeodesics has to be connected to a $P_i$-component of the other one. This means that every sufficiently long $P_i$-component of the image $\iota(w)$ of some word $w\in(Y\cup\mathcal{O})^*$ is connected to a $P_i$-component of any geodesic word in $(X\cup\mathcal{P})^*$ representing the same element, provided that $\iota(w)$ is without backtracking. It is therefore enough to verify that all geodesic words in $(Y\cup\mathcal{O})^*$ with short $O_j$-components are mapped to uniformly quasigeodesic words.

This is made more precise and proved in Lemma~\ref{lem:2}, but the proof requires a way to ensure that the canonical map $\iota$ maps geodesic words to words without backtracking. It turns out that this backtracking can only occur, if there are short words as described in Lemma~\ref{lem:enlarge parabolics}.

\begin{lemma}\label{lem:1}
Let $\OO=\{O_1,\ldots,O_m\}$ be a peripheral structure of $H$ consisting of subgroups $O_j\subseteq H\cap P_{i_j}^{g_j}$ for some $1\leq i_j\leq n$ and $g_j\in X^*$.

Let $L,\lambda\geq1$, $c\geq0$ and $\mu:=\max(\{\abs{y}\mid y\in Y\}\cup\{2\abs{g_j}+1\mid1\leq j\leq m\})$. Let $\varepsilon=\varepsilon(\lambda,c+\mu)$ be as in the conclusion of Theorem~\ref{thm:bcp} and
\[
D\geq\max_{1\leq j\leq m}\left(\Dist_{(G,X)}^{(O_j,Y)}(\varepsilon+2\abs{g_j})\right).
\]

Suppose that
\begin{enumerate}
\item for any geodesic word $w\in(Y\cup\mathcal{O})^*$ with $\abs{w}\leq L$, for which $\iota(w)$ is without backtracking, $\iota(w)$ is a $(\lambda,c)$-quasi-geodesic word in $(X\cup\mathcal{P})^*$, and

\item there is some geodesic word $w\in(Y\cup\mathcal{O})^*$ with $\abs{w}\leq L$ such that $\iota(w)$ has backtracking.
\end{enumerate}
Then there is $h\in H\cap g_j^{-1}P_ig_k$ with $i=i_j=i_k$, $h\notin O_j$ if $j=k$, and $\abs{h}_Y\leq D\lambda(\mu+c)$

\begin{proof}
Let $w\in(Y\cup\mathcal{O})^*$ be a geodesic word with $\abs{w}\leq L$. Suppose the path $v$ in $\Gamma(G,X\cup\mathcal{P})$ starting at $1$ and labeled by $\iota(w)$ has backtracking. Then there is a subword $o_1ho_2$ of $w$ with $h\in (Y\cup\mathcal{O})^*$, $o_1\in O_j$, $o_2\in O_k$ and $i:=i_j=i_k$ such that the $P_i$-components $p_1$ and $p_2$ of the subword
\[
\iota(o_1ho_2)=g_j^{-1}p_1g_j\,\iota(h)\,g_k^{-1}p_2g_k
\]
of $\iota(w)$ are connected and that $\iota(h)$ is without backtracking and has no $P_i$-component connected to $p_1$ and $p_2$. If $j=k$, then $h\notin O_j$ since $w$ is geodesic and therefore without backtracking.

Since $p_1$ and $p_2$ are connected, $g_jhg_k^{-1}$ represents some element $p\in P_i$. As a subword of $w$, $h$ is geodesic with $\abs{h}\leq L$ and since $\iota(h)$ is without backtracking, it is $(\lambda,c)$-quasigeodesic by assumption 1. Hence:
\[
\abs{h}_{Y\cup\mathcal{O}}\leq\abs{\iota(h)}\leq \lambda\abs{h}_{X\cup\mathcal{P}}+\lambda c=\lambda\abs{g_j^{-1}pg_k}_{X\cup\mathcal{P}}+\lambda c\leq \lambda(\mu+c)
\]
Now $p^{-1}g_j\iota(h)g_k^{-1}$ is a $(\lambda,c+\mu)$-quasigeodesic without backtracking and representing the trivial element. The choice of $\varepsilon$ implies that all of its $\PP$-components must have a $d_X$-length of at most $\varepsilon$. Since every $O_j$-component $o$ of $h$ is mapped under $\iota$ to $g_j^{-1}pg_j$ for some $P_{i_j}$-component $p$ of $\iota(h)$, it follows that $\abs{o}_X\leq\varepsilon+2\abs{g_j}$. The choice of $D$ then implies $\abs{o}_Y\leq D$, which in turn implies
\[
\abs{h}_Y\leq D\lambda(\mu+c).
\]
\end{proof}
\end{lemma}

It is now possible to prove the following lemma, which states that under the right conditions it suffices to check whether short geodesic words with short $\OO$-components in $(Y\cup\mathcal{O})^*$ are mapped to quasigeodesic words in $(X\cup\mathcal{P})^*$ to verify that the same is true for short geodesic words with arbitrarily large $\OO$-components.

\begin{lemma}\label{lem:2}
Let $\OO=\{O_1,\ldots,O_m\}$ be a peripheral structure of $H$ consisting of subgroups $O_j\subseteq H\cap P_{i_j}^{g_j}$ for some $1\leq i_j\leq n$ and $g_j\in X^*$.

Let $\mu:=\max(\{\abs{y}\mid y\in Y\}\cup\{2\abs{g_j}+1\mid1\leq j\leq m\})$, $N\geq2(\mu+1)^2$ and $L\geq1$. Let $\varepsilon=\varepsilon(N,N+\mu)$ be as in the conclusion of Theorem~\ref{thm:bcp}, $C:=\sqrt{\frac{N}{2}}-\mu$ and
\[
D\geq\max_{1\leq j\leq m}\left(\Dist_{(G,X)}^{(O_j,Y)}(\varepsilon+2\abs{g_j})\right).
\]

Suppose that
\begin{enumerate}
\item every geodesic word $w\in(Y\cup\mathcal{O})^*$ with $\abs{w}\leq L$ and $\abs{o}_Y\leq D$ for every letter $o\in\mathcal{O}$ of $w$ is mapped to a $(C,C)$-quasigeodesic word in $(X\cup\mathcal{P})^*$ under $\iota$, and

\item there is no $h\in H\cap g_j^{-1}P_ig_k$ with $i=i_j=i_k$, $h\notin O_j$ if $j=k$, and $\abs{h}_Y\leq DN(\mu+N)$.
\end{enumerate} 
Then for every geodesic $w\in(Y\cup\mathcal{O})^*$ with $\abs{w}\leq L$ the word $\iota(w)$ is a $(N,N)$-quasigeodesic word in $(X\cup\mathcal{P})^*$.

\begin{proof}
Let $w\in(Y\cup\mathcal{O})^*$ geodesic with $\abs{w}\leq L$ and assume that the path $v'$ in $\Gamma(G,X\cup\mathcal{P})$ starting at $1$ and labeled by $\iota(w)$ is without backtracking.

Let $o\in O_j$ be a letter of $w$ with $\abs{o}_Y\geq D$. Hence, $o=g_j^{-1}pg_j$ for some $p\in P_{i_j}$. It follows that
\[
\Dist_{(G,X)}^{(O_j,Y)}(\abs{o}_X)\geq\abs{o}_Y\geq D\geq\Dist_{(G,X)}^{(O_j,Y)}(\varepsilon+2\abs{g_j}).
\]
This implies that $\abs{p}_X\geq\abs{o}_X-2\abs{g_j}\geq\varepsilon$.

Since $\iota(w)$ is without backtracking, it follows from Theorem \ref{thm:bcp} that a geodesic $v$ in $\Gamma(G,X\cup\mathcal{P})$ from $1$ to $w$ must have a $P_{i_j}$-component connected to the $P_{i_j}$-component $p$ of $\iota(w)$.

If $w$ has at least $\frac{\abs{\iota(w)}}{N}-(N-2\mu)$ of such long $\OO$-components, i.e.\ $\OO$-components of $d_X$-length at least $\varepsilon+2\abs{g_j}$, then $v$ must have at least the same number of $\PP$-components and therefore:
\[
\abs{w}_{X\cup\mathcal{P}}=\abs{v}\geq\frac{\abs{\iota(w)}}{N}-(N-2\mu)
\]

Suppose now that $w$ has $k\leq\frac{\abs{\iota(w)}}{N}-(N-2\mu)$ long $\OO$-components. Then $w$ can be written as $w_1o_1w_2o_2\ldots o_kw_{k+1}$, where the $o_l$ are the long $\OO$-components of $w$ and $w_1,\ldots,w_{k+1}\in(Y\cup\mathcal{O})^*$ are the (possibly empty) subwords of $w$ between them. Each $o_l$ is mapped under $\iota$ to some word $g_{j_l}^{-1}p_lg_{j_l}$, where $p_l\in P_{i_{j_l}}$ is a $\PP$-component of $\iota(w)$. $v$ must then have $\PP$-components $q_1,\ldots,q_k$, where each $q_l$ is connected to the $\PP$-component $p_l$ of $\iota(w)$. Using Theorem~\ref{thm:bcp} and the choice of $\varepsilon$, it is easy to show that $q_1,\ldots,q_k$ appear in $v$ in order, i.e.\ that $v$ is of the form $v_1q_1v_2q_2\ldots q_kv_{k+1}$. This implies that for each $l\in\{1,\ldots,k+1\}$
\[
\abs{v_l}_{x\cup\mathcal{P}}\geq\abs{w_l}_{X\cup\mathcal{P}}-2(\max\{\abs{g_j}\mid1\leq j\leq m\}+1).
\]

Since the $w_l$ are shorter than $w$ and contain no long $\OO$-components, their images $\iota(w_l)$ are $(C,C)$-quasigeodesic by assumption 1. It follows that:
\begin{align*}
\abs{w}_{X\cup\mathcal{P}} &= \abs{v}=\sum_{l=1}^{k+1}\abs{v_l}_{X\cup\mathcal{P}}+k\\
&\geq\sum_{l=1}^{k+1}\abs{w_l}_{X\cup\mathcal{P}}-2k(\max\{\abs{g_j}\mid1\leq j\leq m\}+1)+k\\
&\geq \sum_{l=1}^{k+1}\left(\frac{1}{C}\abs{\iota(w_l)}-C\right)-k\mu\\
&= \frac{1}{C}\left(k\mu+\sum_{l=1}^{k+1}\abs{\iota(w_l)}\right)-\frac{k\mu}{C}-(k+1)C-k\mu\\
&\geq \frac{1}{C}\abs{\iota(w)}-\frac{k\mu}{C}-(k+1)C-k\mu\\
&= \frac{1}{C}\abs{\iota(w)}-\left(C+k\left(C+\left(1+\frac{1}{C}\right)\mu\right)\right)\\
&\geq \frac{1}{C}\abs{\iota(w)}-\left(C+\left(\frac{\abs{\iota(w)}}{N}-(N-2\mu)\right)(C+2\mu)\right)\\
&= \left(\frac{1}{C}-\frac{1}{N}(C+2\mu)\right)\abs{\iota(w)}-\left(C-(N-2\mu)(C+2\mu)\right)
\end{align*}

Since $C\leq N-2\mu$ and $2C(C+\mu)\leq N$, it follows again that:
\[
\abs{w}_{X\cup\mathcal{P}}\geq\frac{\abs{\iota(w)}}{N}-(N-2\mu)
\]

Now let $w'$ be a non-trivial subword of $\iota(w)$. Let $s$ be minimal, such that $w'$ is a subword of $\iota(w_r\ldots w_{r+s})$ for some $r\in\{1,\ldots,k+1\}$. It follows that:
\[
\abs{w'}_{X\cup\mathcal{P}}\geq\abs{\iota(w_r\ldots w_{r+s})}_{X\cup\mathcal{P}}-2\mu\geq\frac{\abs{\iota(w_r\ldots w_{r+s})}}{N}-N\geq\frac{\abs{w'}}{N}-N
\]
Hence, $\iota(w)$ is $(N,N)$-quasigeo\-desic.

\medskip
Suppose now that $w\in(Y\cup\mathcal{O})^*$ is any geodesic word with $\abs{w}\leq L$. It follows by Lemma~\ref{lem:1} from assumption 2.\ and the above that $\iota(w)$ is without backtracking. This in turn implies that $\iota(w)$ is itself $(N,N)$-quasigeodesic.
\end{proof}
\end{lemma}

The rest of this section is dedicated to the proof of the main theorem:

\setcounter{section}{1}
\setcounter{subsection}{0}
\begin{theorem}\label{thm:main thm}
Suppose that for all $i\in\{1,\ldots,n\}$ the following are given:
\begin{enumerate}
\item a solution to the membership problem of $(P_i,O)$ for every finitely generated subgroup $O\leq P_i$ and
\item an algorithm, which decides whether a given finitely generated subgroup of some $P_i$ is finite.
\end{enumerate}

Then there is a semi-algorithm which stops, if and only if $H$ is a relatively quasiconvex subgroup of $G$, and if it stops, it returns
\begin{itemize}
\item generating sets of an induced structure $\OO$ on $H$, and
\item $\lambda\geq1,c\geq0$, such that the inclusion $(H,d_{Y\cup\mathcal{O}})\to(G,d_{X\cup\mathcal{P}})$ is a $(\lambda,c)$-quasiisometric embedding.
\end{itemize}
\end{theorem}
\setcounter{section}{4}
\setcounter{subsection}{3}

We will first prove an auxiliary lemma:
\begin{lemma}
In the situation of Theorem~\ref{thm:main thm}, the following operations can be executed in finite time:
\begin{enumerate}
\item[(O1)] For $g\in G$, finitely generated $O\leq H\cap P_i^g$ and $n\in\N$ compute an upper bound for $\Dist_{(G,X)}^{(O,Y)}(n)$.

\item[(O2)] For $w\in X^*$ and finitely generated $O\leq H\cap P_i^g$ decide, whether $w$ represents an element in $O$.

\item[(O3)] For $C\geq1$ and $w\in(X\cup\mathcal{P})^*$ decide, whether $w$ is a $(C,C)$-quasigeodesic word.
\end{enumerate}

\begin{proof} By assumption 1., a solution to the word problem in each $P_i$ is given. By Zhang (\cite{Zhang2018},~Thm.~4.1.12) it is therefore possible to construct a solution to the word problem in $G$.

\begin{enumerate}
\item[(O1)] Let $g\in G$, $O\leq H\cap P_i^g$ with finite generating set $S\subseteq Y^*$ and $k\in\N$.

By Theorem~\ref{thm:dehn+delta} an upper bound for the relative Dehn function of $(G,\PP)$ can be computed. It follows from Theorem~\ref{thm:undistorted parabolics}, that it is possible to compute $\Dist_{(G,X)}^{(P_i,X_i)}$. Since a solution to the membership problem of $(P_i,O^{g^{-1}})$ is given, it is by Theorem~\ref{thm:memb prob<->distortion} possible to construct an algorithm that computes $\Dist_{(P_i,X_i)}^{(O^{g^{-1}},S^{g^{-1}})}$. So by Lemma~\ref{lem:distortion facts} and Lemma~\ref{lem:chained distortion}:
\begin{align*}
\Dist_{(G,X)}^{(O,Y)}(k) &\leq \max\{\abs{s}_Y\mid s\in S\}\cdot\Dist_{(G,X)}^{(O,S)}(k)\\
&\leq \max\{\abs{s}_Y\mid s\in S\}\cdot\Dist_{(G,X)}^{(O^{g^{-1}},S^{g^{-1}})}(k+2\abs{g})\\
&\leq \max\{\abs{s}_Y\mid s\in S\}\cdot\left(\Dist_{(P_i,X_i)}^{(O^{g^{-1}},S^{g^{-1}})}\circ\Dist_{(G,X)}^{(P_i,X_i)}\right)(k+2\abs{g})
\end{align*}
Therefore, an upper bound for $\Dist_{(G,X)}^{(O,Y)}(k)$ can be computed.

\item[(O2)] As explained above, it is possible to construct a solution to the word problem in $G$. Since $\Dist_{(G,X)}^{(O,Y)}$ is computable, it is by Theorem~\ref{thm:memb prob<->distortion} possible to construct a solution to the membership problem for $(G,O)$.

\item[(O3)] It suffices to show that it is possible for any $w\in(X\cup\mathcal{P})^*$ to compute $\abs{w}_{X\cup\mathcal{P}}$.

Since $\Dist_{(G,X)}^{(P_i,X_i)}$ can be computed for all $i\in\{1,\ldots,n\}$, the membership problem of $(G,P_i)$ is solvable by Theorem~\ref{thm:memb prob<->distortion}. Hence, it is possible to construct a word $w'\in(X\cup\mathcal{P})^*$ representing the same element as $w$ which is without backtracking.

Let $\varepsilon(\abs{w'},\abs{w'})$ be as in the conclusion of Theorem~\ref{thm:bcp} and $v\in(X\cup\mathcal{P})^*$ a geodesic word representing the same element as $w$. Since $v$ and $w'$ are obviously $(\abs{w'},\abs{w'})$-quasigeodesic, it follows from Theorem~\ref{thm:bcp} that the $\abs{\cdot}_X$-length of the $\PP$-components of $v$ is bounded by:
\[
\max\{\abs{p}_X\mid p\text{ is a }\PP\text{-component of }w'\}+2\varepsilon
\]

Since solutions to the membership problem for $(G,P_i)$ and the word problem for $P_i$ are given, it is possible to find a shortest word in $(X\cup\mathcal{P})^*$ with $\PP$-components of bounded length as above, which represents the same element as $w$. The length of this word is then $\abs{w}_{X\cup\mathcal{P}}$.
\end{enumerate}
\end{proof}
\end{lemma}

\begin{proof}[of Theorem~\ref{thm:main thm}]
The algorithm is as follows:

Use the algorithm of Theorem~\ref{thm:dehn+delta} to compute a hyperbolicity constant $\delta$ for the Cayley graph $\Gamma(G,X\cup\mathcal{P})$.

Begin enumerating all $m\in\N$ and $g_j\in G$, $i_j\in\{1,\ldots,n\}$ and finite sets $Y_j\subseteq H\cap P_{i_j}^{g_j}$ for all $j\in\{1,\ldots,m\}$. This is possible, since (O2) provides a solution to the membership problem of $(G,P_i)$ for all $i\in\{1,\ldots,n\}$, and hence the set of all triples $(h,g,i)\in H\times G\times\{1,\ldots,n\}$ with $ghg^{-1}\in P_i$ can be computed.

While enumerating these sets, run the following partial algorithm for each new set in parallel:

\begin{itemize}
\item Step 1:

Let $\mu:=\max(\{\abs{y}\mid y\in Y\}\cup\{2\abs{g_j}+1\mid1\leq j\leq m\})$ and $N:=\lceil2(\mu+1)^2\rceil$. Continue with step~2.

\item Step 2:

There are computable $L>0$ and $\lambda\geq1,c\geq0$ only depending on $\delta$ and $N$, such that every $(L,N,N)$-local-quasigeodesic in $\Gamma(G,X\cup\mathcal{P})$ is a $(\lambda/\mu,c)$-quasigeodesic. (cf.~Coornaert~et~al.~\cite{Coornaert1990}, Ch.~3, Thm.~1.4)

Let $C:=\sqrt{\frac{N}{2}}-\mu$ and $\varepsilon=\max\{\varepsilon(\lambda/\mu,c+\mu),\varepsilon(N,N+\mu)\}$ as in the conclusion of Theorem~\ref{thm:bcp}. Using (O1), find an upper bound
\[
D\geq\max_{1\leq j\leq m}\left(\Dist_{(G,X)}^{(O_j,Y)}(\varepsilon+2\abs{g_j})\right)
\]
Continue with step~3.

\item Step 3:

Utilizing (O2), check for all $j,k\in\{1,\ldots,m\}$ with $i:=i_j=i_k$ and for all $h\in Y^*$ with $\abs{h}\leq\max\left\{DN(\mu+N),D(\lambda/\mu)(\mu+c)\right\}$, whether $g_jhg_k^{-1}\in P_i$ and $h\notin O_j$ if $j=k$.

If no such $h$ is found, continue with step~4, else terminate this instance of the partial algorithm.

\item Step 4:

Check with (O3) for every geodesic word $w\in(Y\cup\mathcal{O})^*$ with $\abs{w}\leq L$ and $\abs{o}_Y\leq D$ for every letter $o\in\mathcal{O}$ of $w$, whether $\iota(w)$ is a $(C,C)$-quasigeodesic word in $(X\cup\mathcal{P})^*$.

If there is a word $w$ such that $\iota(w)$ is not $(C,C)$-quasigeodesic, increase $N$ by one and repeat from step~2, else continue with step~5.

\item Step 5:

Using the algorithm from assumption 2., check for every $j\in\{1,\ldots,m\}$, whether $O_j$ is finite. Eliminate all finite peripheral subgroups from $\OO$.

Terminate the whole algorithm and return the constructed peripheral structure $\OO$ of $H$ as well as $\lambda$ and $c$.
\end{itemize}

\medskip
First, assume this algorithm terminates. The elimination of all finite peripheral subgroups in step~5 does not change that, as determined in step~4, all geodesic words $w\in(Y\cup\mathcal{O})^*$ with $\abs{w}\leq L$ and short $\OO$-components are mapped to $(C,C)$-quasigeodesics under $\iota$. So every geodesic word $w\in(Y\cup\mathcal{O})^*$ with $\abs{w}\leq L$ is according to Lemma~\ref{lem:2} mapped to a $(N,N)$-quasigeodesic word in $(X\cup\mathcal{P})^*$.

Therefore, every geodesic word $w\in(Y\cup\mathcal{O})^*$ is mapped under $\iota$ to a $(L,N,N)$-locally-quasigeodesic word. Because of the choice of $L$ and $N$, the inclusion $(H,d_{Y\cup\mathcal{O}})\to(G,d_{X\cup\mathcal{P}})$ is a $(\lambda,c)$-quasiisometric embedding. In particular, $H$ is relatively quasiconvex in $(G,\PP)$ by Theorem~\ref{thm:qc-embedding->rel qc}.

It remains to show that $\OO$ is an induced structure. Let $i\in\{1,\ldots,n\}$ and $g\in X^*$ such that $H\cap P_i^g$ is infinite. If $h\in H\cap P_i^g$, then:
\[
\abs{h}_{Y\cup\mathcal{O}}\leq\lambda\abs{h}_{X\cup\mathcal{P}}+\lambda c\leq\lambda(2\abs{g}+1)+\lambda c
\]
Since $H\cap P_i^g$ is infinite but contained in a ball of finite radius in $\Gamma(H,Y\cup\mathcal{O})$, there has to be a geodesic word $h=h_1oh_2\in(Y\cup\mathcal{O})^*$, representing an element $g^{-1}pg\in H\cap P_i^g$, which contains an $O_j$-component $o=g_j^{-1}p'g_j$ of $\abs{\cdot}_Y$-length at least $D$. The $P_{i_j}$-component $p'$ of $\iota(h)=\iota(h_1)g_j^{-1}p'g_j\iota(h_2)$ therefore has $\abs{\cdot}_X$-length at least $\varepsilon$. 

Hence, by Theorem~\ref{thm:bcp}, $i=i_j$ and the $P_i$-components $p$ of $g^{-1}pg\in(X\cup\mathcal{P})^*$ and $p'$ of $\iota(h)$ are connected. This means that $p'':=g_jh_1^{-1}g^{-1}\in P_i$ and therefore:
\[
O_j^{h_1^{-1}}\subseteq(H\cap P_i^{g_j})^{h_1^{-1}}=H^{h_1^{-1}}\cap P_i^{p''g}=H\cap P_i^g
\]
It now suffices to show that $O_j=H\cap P_i^{g_j}$ for all $1\leq j\leq m$.

Since geodesic words of arbitrary length are mapped to $(\lambda/\mu,c)$-quasigeodesic words under $\iota$ and since there is no $h\in H\cap P_i^{g_j}\setminus O_j$ with $\abs{h}_Y\leq D\frac{\lambda}{\mu}(\mu+c)$, $\iota$ must map all geodesic words to words without backtracking by Lemma~\ref{lem:1}. This implies that there can be no $h\in H\cap P_i^{g_j}\setminus O_j$ at all. Therefore $H\cap P_i^{g_j}=O_j$.

\medskip
Conversely it is easy to see, that the algorithm will terminate if $H$ is relatively quasiconvex in $G$:

Since $(G,\PP)$ is relatively hyperbolic, the algorithm of Dahmani from Theorem~\ref{thm:dehn+delta} will terminate. Since all subgroups in the induced structure of a finitely generated relatively quasiconvex subgroup are finitely generated (Osin \cite{Osin2006}, Proposition~2.29), it is also clear that enumerating the sets on which to run the partial algorithm will eventually produce an induced structure, if $H$ is relatively quasiconvex. For an induced structure the inclusion $(H,d_{Y\cup\mathcal{O}})\to(G,d_{X\cup\mathcal{P}})$ is a quasiisometric embedding (cf.~Hruska~\cite{Hruska2010}, Thm.~10.1), hence the partial algorithm will terminate for such a structure.
\end{proof}

\section{Discussion}\label{sec:Discussion}

This section first discusses the necessity of the main theorem's assumptions and to what extend it might be possible to relax or even omit them. Secondly, a more constructive approach to the algorithm is sketched and evaluated.

\smallskip
The given algorithm relies strongly on the solutions to the membership problems provided by assumption~$1$. In particular step~3 of the partial algorithm depends on these solutions, since it searches for parabolic elements of $H$, which are not accounted for by the peripheral structure. A meaningful relaxation of these conditions, if at all possible, would therefore at least require a fundamental change of the algorithm.

In contrast, it is possible to prove a version of the main theorem, which does not require the decidability of the finiteness of finitely generated parabolic subgroups. In this version of the algorithm the elimination of the finite peripheral subgroups in step~5 is dropped. The consequence of this modification is that the returned peripheral structure will consist of an induced structure together with a collection of finite groups.

It is also easy to see that assumption~$2$.\ is a necessary condition to eliminate the finite subgroups from the induced structure:
\newline Let $G$ be a group with solvable word problem and $H$ a finitely generated subgroup of $G$. $H$ is relatively quasiconvex in $(G,\{G\})$ and the induced structure of $(G,\{G\})$ on $H$ is $\{H\}$ if $H$ is infinite and $\emptyset$ if $H$ is finite. Therefore any algorithm determining the induced structure of $H$ also determines whether $H$ is finite.

\smallskip
The given algorithm relies on the ability to verify, whether a given peripheral structure of $H$ is an induced structure and then systematically checks every possible peripheral structure until it finds an induced one. It is also possible to take a more constructive approach to finding an induced structure. Instead of terminating the partial algorithm in step~3 if an $h$ fulfilling the conditions is found, it would be possible to modify the peripheral structure according to Lemma~\ref{lem:enlarge parabolics}, i.e.\ adding $h$ to $Y_j$ if $j=k$ and replacing $Y_j$ and $Y_k$ by $Y_j^h\cup Y_k$ if $j\neq k$, and then repeat from step~2.

Through this process the peripheral subgroups will be replaced by pairwise non-conjugate maximal parabolic subgroups eventually. However, this new peripheral structure does not have to include an induced structure, since it is not clear at the beginning of this process, that for every infinite maximal parabolic subgroup of $H$ there is a peripheral subgroup in $\OO$ which can in $H$ be conjugated into this maximal parabolic subgroup.

A more constructive approach would therefore require the addition of new groups to the peripheral structure, as well as the modification of peripheral groups as described above. Without the ability to verify that a given peripheral structure is not an induced structure this still requires running parallel partial algorithms. Hence it is in general not clear which approach is the more efficient for a concrete implementation of the algorithm.

The algorithm given in the proof above was chosen because it requires less careful bookkeeping than the constructive approach and is therefore more comprehensible.

\newpage
\addcontentsline{toc}{section}{References}
\bibliography{References} 
\bibliographystyle{ieeetr}

\end{document}